\documentclass[11pt]{amsart}
\usepackage{amsthm, amsmath}
\usepackage{amssymb}
\usepackage[all]{xy}

\theoremstyle{plain}
\newtheorem{theorem}{Theorem}[section]
\newtheorem{lemma}[theorem]{Lemma}
\newtheorem{cor}[theorem]{Corollary}
\newtheorem{prop}[theorem]{Proposition}

\theoremstyle{remark}

\theoremstyle{definition}

\newtheorem{defn}[theorem]{Definition}
\newtheorem{Remark}[theorem]{Remark}
\numberwithin{equation}{subsection}
\numberwithin{theorem}{subsection}

\newcommand{\calH}{\mathcal H}

\newcommand{\calO}{\mathcal O}

\newcommand\hfld[2]{\smash{\mathop{\hbox to 10mm{\rightarrowfill}}
     \limits^{\scriptstyle#1}_{\scriptstyle#2}}}
\newcommand\hflg[2]{\smash{\mathop{\hbox to 10mm{\leftarrowfill}}
     \limits^{\scriptstyle#1}_{\scriptstyle#2}}}

\title{The Satake isomorphism for special maximal \\ parahoric Hecke algebras}
\author{Thomas J. Haines* and Sean Rostami}

\begin{document}

\thanks{*Research partially supported by NSF Focused Research Grant DMS-0554254 and NSF Grant DMS-0901723, and by a University of Maryland GRB Semester Award.}
\subjclass[2000]{Primary 11E95, 20G25; Secondary 22E20}

\begin{abstract} 
Let $G$ denote a connected reductive group over a nonarchimedean local field $F$.  Let $K$ denote a special maximal parahoric subgroup of $G(F)$.  We establish a Satake isomorphism for the Hecke algebra $\calH_K$ of $K$-bi-invariant compactly supported functions on $G(F)$.  The key ingredient is a Cartan decomposition describing the double coset space $K\backslash G(F)/K$.  As an application we define a transfer homomorphism $t: \calH_{K^*}(G^*) \rightarrow \calH_K(G)$ where $G^*$ is the quasi-split inner form of $G$.  We also describe how our results relate to the treatment of Cartier \cite{Car}, where $K$ is replaced by a special maximal compact open subgroup $\widetilde{K} \subset G(F)$ and where a Satake isomorphism is established for the Hecke algebra $\calH_{\widetilde{K}}$.
\end{abstract}

\maketitle

%%% Section 1
\markboth{\small{T. Haines} and  \small{S. Rostami}}
{\small{Satake isomorphism for special maximal parahoric Hecke algebras}}

\section{Introduction}

The Satake isomorphism plays an important role in automorphic forms and in representation theory of p-adic groups.  For global applications, one may often work with unramified groups.  We begin by recalling the Satake isomorphism in this context.  Let $G$ denote an unramified group over a nonarchimedean local field $F$.  Let $v_F$ denote a special vertex in the Bruhat-Tits building $\mathcal B(G_{\rm ad}(F))$.   Let $\widetilde{K} = \widetilde{K}_{v_F}$ denote a special maximal compact open subgroup of $G(F)$ which fixes $v_F$.  Let 
$$
\calH_{\widetilde{K}} = C^\infty_c(\widetilde{K} \backslash G(F)/\widetilde{K})
$$
denote the Hecke algebra of $\widetilde{K}$-bi-invariant compactly-supported complex-valued functions on $G(F)$.  Let $A$ denote a maximal $F$-split torus in $G$ whose corresponding apartment in $\mathcal B(G_{\rm ad}(F))$ contains $v_F$.  Let $W = W(G,A)$ denote the relative Weyl group.  Then the Satake isomorphism is a $\mathbb C$-algebra isomorphism 
\begin{equation*}
\calH_{\widetilde{K}} ~ \widetilde{\rightarrow} ~ \mathbb C[X_*(A)]^{W}.
\end{equation*}
(See \cite{Car}.)  A key ingredient is the Cartan decomposition
$$
\widetilde{K} \backslash G(F)/\widetilde{K} ~ \cong ~ W(G,A)\backslash X_*(A).
$$

Now let $G$ denote an arbitrary connected reductive group over $F$ and let $\widetilde{K}, v_F$ and so on have the same meaning as above.  A form of the Satake isomorphism for such $G$ was described by Cartier \cite{Car}, but it is less explicit than that above.  It identifies $\calH_{\widetilde{K}}$ with the ring of functions $$\mathbb C[M(F)/M(F)^1]^{W},$$ where $M := {\rm Cent}_G(A)$ is a minimal $F$-Levi subgroup of $G$ and $M(F)^1$ is the unique maximal compact open subgroup of $M(F)$.  The quotient $M(F)/M(F)^1$ is a free abelian group $\widetilde{\Lambda}_M$ which contains $X_*(A)$ and has the same rank.  (In \cite{Car}, our $\widetilde{\Lambda}_M$ is denoted $\Lambda(M)$ or simply $\Lambda$.)  As Cartier explains, in this general context we have a Satake isomorphism
$$
\calH_{\widetilde{K}} ~ \cong ~ \mathbb C[\widetilde{\Lambda}_M]^W,
$$
and a Cartan decomposition
$$
\widetilde{K} \backslash G(F)/\widetilde{K} ~ \cong ~ W(G,A)\backslash \widetilde{\Lambda}_M.
$$
However, Cartier does not identify $\widetilde{\Lambda}_M$ explicitly, except in special cases. 

Now let $K = K_{v_F}$ denote the special maximal parahoric subgroup of $G(F)$ corresponding to $v_F$; it is a normal subgroup of $\widetilde{K}_{v_F}$ having finite index (see section \ref{Ktilde_char_sec}).  This paper concerns the Hecke algebra $\calH_K = C^\infty_c(K \backslash G(F)/K)$.  In several situations, it is more appropriate to consider $\calH_K$ instead of $\calH_{\widetilde K}$, for example in relation to Shimura varieties having parahoric level structure (see \cite{Rap} and \cite{H05}).

Let $M(F)_1 \subset M(F)$ denote the unique parahoric subgroup of $M(F)$; it is a finite-index normal subgroup of $M(F)^1$.  Our main result is the following theorem.

\begin{theorem} \label{Sat_thm}
Let $\Lambda_M := M(F)/M(F)_1$.  There is a canonical isomorphism 
\begin{equation*}
\calH_K ~ \widetilde{\rightarrow} ~ \mathbb C[\Lambda_M]^{W}.
\end{equation*}  
The group $\Lambda_M$ is a finitely generated abelian group which can be explicitly described and which has the property that $\widetilde{\Lambda}_M = \Lambda_{M}/torsion$.  Moreover, $\widetilde{K}/K \cong \Lambda_{M,\rm tor}$, the torsion subgroup of $\Lambda_M$.
\end{theorem}

When $G$ is unramified over $F$ or when $G$ is semi-simple and simply connected, it turns out that $\widetilde{K} = K$ and $\widetilde{\Lambda}_M \cong \Lambda_M$ (see section \ref{structure_sec}) so that our theorem does not give any new information in those cases.  However our results are new in case $\widetilde{K} \neq K$, and different methods from \cite{Car} are needed to prove them.  For ramified groups in particular, our results are expected to play some role in the study of Shimura varieties with parahoric level structure at $p$.  For more about ramified groups and Shimura varieties with parahoric level the reader should consult \cite{Rap}, \cite{PR}, and \cite{Kr}. 

In order to describe $\Lambda_M$, we need to recall some notation and results of Kottwitz \cite{Ko97}.  Let $F^s$ denote a separable closure of $F$, and let $F^{\rm un}$ denote the maximal unramified extension of $F$ in $F^s$.  Let $L = \widehat{F^{\rm un}}$ denote the completion of $F^{\rm un}$ with respect to the valuation on $F^{\rm un}$ which extends the normalized valuation on $F$.  Let $I = {\rm Gal}(F^s/F^{\rm un}) \cong {\rm Gal}(L^s/L)$ denote the inertia subgroup of ${\rm Gal}(F^s/F)$, and let $\sigma \in {\rm Aut}(L/F)$ denote the Frobenius automorphism.  In \cite{Ko97} Kottwitz defined a surjective homomorphism
$$
\kappa_G: G(L) \rightarrow X^*(Z(\widehat{G}))_I,
$$
and in loc.~cit. $\S 7.7$ he also proved that this induces a surjective homomorphism
$$
\kappa_G: G(F) \rightarrow X^*(Z(\widehat{G}))^\sigma_I
$$
of the groups of $\sigma$-invariants.  Set $G(L)_1 := {\rm ker}(\kappa_G)$ and $G(F)_1 := G(F) \cap G(L)_1$.  (When $G = M$, this is consistent with our definition of $M(F)_1$ above, see Lemmas \ref{parahoric_cap_Levi},  \ref{para_cap_min_Levi}.)

The Iwahori-Weyl group  $\widetilde{W}$ for $G$ carries a natural action under $\sigma$ and contains a $\sigma$-invariant abelian subgroup $\Omega_G$ (the subgroup of {\em length-zero elements}).  By choosing representatives in the normalizer of $A$ we may embed $\widetilde{W}^\sigma$ set-theoretically into $G(F)$, and then $\Omega_G^\sigma$ is mapped by $\kappa_G$ isomorphically onto $X^*(Z(\widehat{G}))^\sigma_I$ (see section \ref{notation_sec}).  The following is the sought-after explicit description of $\Lambda_M$:

\begin{prop}
The Kottwitz homomorphism induces an isomorphism $$\Lambda_M = M(F)/M(F)_1 \cong X^*(Z(\widehat{M}))_I^\sigma.$$  We can also identify $\Lambda_M$ with $\Omega_M^\sigma$ via the Kottwitz isomorphism $\kappa_M: \Omega_M^\sigma ~ \widetilde{\rightarrow} ~ X^*(Z(\widehat{M}))_I^\sigma$.
\end{prop}

As before, the main step in the proof of Theorem \ref{Sat_thm} is an appropriate Cartan decomposition.  

\begin{theorem} \label{Cart_decomp_stmt}
The embedding $\Omega_M^\sigma \subset \widetilde{W}^\sigma \hookrightarrow G(F)$ determines a bijection
$$
W(G,A) \backslash \Omega_M^\sigma \cong K \backslash G(F) /K.
$$
Equivalently, via the isomorphism $\kappa_M: \Omega^\sigma_M ~\widetilde{\rightarrow} ~ 
X^*(Z(\widehat{M}))^\sigma_I$, we have a bijection
$$
W(G,A) \backslash X^*(Z(\widehat{M}))^\sigma_I ~ \widetilde{\rightarrow} ~ K \backslash G(F)/K.
$$
\end{theorem}

We give additional information about the finitely generated abelian group $\Lambda_M$ in section \ref{structure_sec}.  For example, we prove that if $G$ is an inner form of a split group, then $\Lambda_M = X^*(Z(\widehat{M})) = X_*(T)_\sigma$ (see Corollary \ref{G*_split_cor}).  

Finally, let $G^*$ denote the quasi-split inner form of $F$, and consider special maximal parahoric subgroups $K^* \subset G^*(F)$ and $K \subset G(F)$.  In section \ref{transfer_sec}, we define a canonical transfer homomorphism $t: \calH_{K^*}(G^*) \rightarrow \calH_K(G)$, and we establish some of its basic properties.  

This article relies heavily on the ideas of Kottwitz, especially as they are manifested in the article \cite{HR}.  The main theorems of \cite{HR} provide the starting points for the proof of Theorem \ref{Cart_decomp_stmt}.

\section{Notation} \label{notation_sec}

\subsection{Ring-theoretic notation}  Let $\mathcal O = \mathcal O_F$ (resp. $\mathcal O_L$) denote the ring of integers in the field $F$ (resp. $L$).  Let $\varpi$ denote a uniformizer of $F$ (resp. $L$), and let $k_F$ denote the residue field of $F$.  We may identify the residue field $k_L$ with an algebraic closure of $k_F$.  Let $\Gamma := {\rm Gal}(F^s/F)$.

Throughout this paper, if $J \subset G(F)$ denotes a compact open subgroup, we make
$$
\calH_J := C^\infty_c(J\backslash G(F)/J)
$$
a convolution algebra by using the Haar measure on $G(F)$ which gives $J$ volume 1.

\subsection{Buildings notation}  Let $\mathcal B(G(L))$ (resp. $\mathcal B(G(F))$) denote the Bruhat-Tits building of $G(L)$ (resp. $G(F)$).  The building $\mathcal B(G(L))$ carries an action of $\sigma$.  By \cite{BT2}, 5.1.25, we have an identification $\mathcal B(G(F)) = \mathcal B(G(L))^\sigma$.  Moreover, there is a bijection ${\mathbf a}_J \mapsto {\mathbf a}_J^\sigma$ from the set of $\sigma$-stable facets in $\mathcal B(G(L))$ to facets in $\mathcal B(G(F))$ (\cite{BT2}, 5.1.28).  This bijection sends alcoves to alcoves (\cite{BT2}, 5.1.14).  It also follows from loc.~cit.~that every $\sigma$-stable facet $\mathbf a_J$ in $\mathcal B(G(L))$ is contained in the closure $\overline{\mathbf a}$ of a $\sigma$-stable alcove $\mathbf a$.

Let $v_F$ denote a special vertex in $\mathcal B(G_{\rm ad}(F))$ (\cite{Tits}, 1.9).  Let $A$ denote a maximal $F$-split torus in $G$ whose corresponding apartment in $\mathcal B(G_{\rm ad}(F))$ contains $v_F$. Let $\mathcal A$ (resp. $\mathcal A_{\rm ad}$) denote the apartment in $\mathcal B(G(F))$ (resp. 
$\mathcal B(G_{\rm ad}(F))$) corresponding to $A$.   Let $V_{G(F)}$ denote the real vector space $X_*(Z(G))_\Gamma \otimes \mathbb R$.  There is a simplicial isomorphism  (\cite{Tits}, 1.2)
$$
\mathcal A \cong \mathcal A_{\rm ad} \times V_{G(F)}.
$$
Therefore, there is a minimal dimensional facet ${\mathbf a}_0^\sigma$ in $\mathcal A$ associated to a $\sigma$-stable facet ${\mathbf a}_0 \subset \mathcal B(G(L))$, such that
$$
{\mathbf a}_0^\sigma \cong \{ v_F \} \times V_{G(F)}.
$$

We consider parahoric (or Iwahori) subgroups in the sense of \cite{BT2}, 5.2.  That is, to a facet ${\bf a}_J \subset \mathcal B(G(L))$ we associate an $\mathcal O_L$-group scheme $\mathcal G^\circ_{{\bf a}_J}$ with connected geometric fibers, whose group of $\mathcal O_L$-points fixes identically the points of ${\bf a}_J$.  We often write $J(L) := \mathcal G^\circ_{{\bf a}_J}(\mathcal O_L)$.  By \cite{BT2}, 5.2, if ${\bf a}_J$ is $\sigma$-stable we get a parahoric subgroup $J(F) := J(L)^\sigma$ in $G(F)$ and this is associated to the facet ${\bf a}_J^\sigma$ in $\mathcal B(G(F))$.  Moreover, every parahoric subgroup of $G(F)$ is of this form for a unique $\sigma$-stable facet ${\bf a}_J$. 

Now fix a $\sigma$-stable alcove ${\bf a}$ whose closure contains ${\bf a}_0$.  Let $I(L)$ (resp. $K(L)$) denote the Iwahori (resp. parahoric) subgroup of $G(L)$ corresponding to the $\sigma$-stable alcove ${\bf a}$ (resp. facet ${\bf a}_0$).   Then $I := I(F) = I(L)^\sigma$ is the Iwahori subgroup of $G(F)$ corresponding to ${\bf a}^\sigma$.  Also, $K := K(F) = K(L)^\sigma$ is a special maximal parahoric subgroup of $G(F)$ corresponding to ${\bf a}_0^\sigma$ (or equivalently, to $v_F$).

\subsection{Weyl groups and Iwahori-Weyl groups}

For a torus $S$ in $G$, let $N_G(S) = {\rm Norm}_G(S)$ denote its normalizer and $C_G(S) = 
{\rm Cent}_G(S)$ its centralizer.  Let $W(G,S) := N_G(S)/C_G(S)$ denote its Weyl group. 

Fix the torus $A$ as before.   From now on, let $S$ be a maximal $L$-split torus that is defined over $F$ and contains $A$ (\cite{BT2}, 5.1.12).  Let $T = C_G(S)$, a maximal torus of $G$ (defined over $F$) since $G_L$ is quasi-split by Steinberg's theorem. 

We need to recall definitions and facts about Iwahori-Weyl groups; we refer the reader to \cite{HR} for details. Let $T(L)_1 = {\rm ker}(\kappa_T)$, a normal subgroup of $N_G(S)(L)$.  
Let $\widetilde{W} := N_G(S)(L)/T(L)_1$ denote the {\em Iwahori-Weyl} group for $G$.  It carries an obvious action of $\sigma$.  Let $\mathcal A_L$ denote the apartment of $\mathcal B(G(L))$ corresponding to $S$, which we may assume contains the alcove ${\bf a}$ we fixed above.  We let $W_{\rm aff}$ denote the {\em affine Weyl group}, which is a Coxeter group generated by the reflections through the walls of ${\bf a}$.  The group $\widetilde{W}$ acts on the set of all alcoves in the apartment of $\mathcal B(G(L))$ corresponding to $S$; let $\Omega_G = \Omega_{G,\bf a}$ denote the stabilizer of ${\bf a}$.   
There is a $\sigma$-equivariant decomposition
$$
\widetilde{W} = W_{\rm aff} \rtimes \Omega_G.
$$
We extend the Bruhat order $\leq$ and the length function $\ell$ from $W_{\rm aff}$ to $\widetilde{W}$ in the obvious way.  We can identify $W_{\rm aff}$ with the Iwahori-Weyl group associated to the pair $G_{\rm sc}, S_{\rm sc}$, where $S_{\rm sc}$ is the pull-back of $(S \cap G_{\rm der})^\circ$ via $G_{\rm sc} \rightarrow G_{\rm der}$.

We can embed $\widetilde{W}$ {\em set-theoretically} into $G(L)$ by choosing a set-theoretic section of the surjective homomorphism $N_G(S)(L) \rightarrow \widetilde{W}$.  Since $T(L)_1 \subset {\rm ker}(\kappa_G)$, we easily see that the restriction of $\kappa_G$ to $\widetilde{W} \hookrightarrow G(L)$ gives a {\em homomorphism}
$$
\kappa_G: \widetilde{W} \rightarrow  X^*(Z(\widehat{G}))_I
$$
which is surjective and $\sigma$-equivariant and whose kernel is $W_{\rm aff}$.   

\section{Cartan decomposition: reduction to the key lemma}

Changing slightly the notation of \cite{HR}, we set
$$
\widetilde{W}_K := (N_G(S)(L)\cap K(L))/T(L)_1.
$$
We write $\widetilde{W}_K^\sigma := (\widetilde{W}_K)^\sigma$.  

Our starting point is the following fact (see \cite{HR}, esp. Remark 9):  the map $K(L)nK(L) \mapsto n \in \widetilde{W}$ induces a bijection
\begin{equation*}
K(L) \backslash G(L) / K(L) \cong \widetilde{W}_K \backslash \widetilde{W} / \widetilde{W}_K,
\end{equation*}
and taking fixed-points under $\sigma$ yields a bijection 
\begin{equation} \label{HR_decomp}
K(F) \backslash G(F) / K(F) \cong \widetilde{W}_K^\sigma \backslash \widetilde{W}^\sigma / \widetilde{W}_K^\sigma.
\end{equation}

The Cartan decomposition follows immediately from the key lemma below, which allows us to describe the right hand side of (\ref{HR_decomp}) in the desired way.  To state this we note that the $\sigma$-stable alcove ${\bf a}$ is contained in a unique $\sigma$-stable alcove ${\bf a}^M$ in the apartment $\mathcal A^M_L \subset \mathcal B(M(L))$ corresponding to $S$.  As before, we define $\Omega_M \subset \widetilde{W}_M$ to be the stabilizer of ${\bf a}^M$ under the action of $\widetilde{W}_M$ on the alcoves in $\mathcal A^M_L$.

\begin{lemma} \label{Cartan_lemma}
\begin{enumerate}
\item[(I)] There is a tautological isomorphism $\widetilde{W}_K^\sigma ~ \widetilde{\rightarrow} ~ W(G,A)$ which allows us to view $W(G,A)$ as a subgroup of $\widetilde{W}^\sigma$.
\item[(II)] There is a decomposition $\widetilde{W}^\sigma = \widetilde{W}^\sigma_M \cdot W(G,A)$, and $W(G,A)$ normalizes $\widetilde{W}^\sigma_M$.   
\item[(III)] We have $W_{M,\rm aff}^\sigma = 1$, and hence because of the $\sigma$-equivariant decomposition
$$
\widetilde{W}_M = W_{M, \rm aff} \rtimes \Omega_M$$
we have $\widetilde{W}^\sigma = \Omega_M^\sigma \rtimes W(G,A).$
\end{enumerate}
\end{lemma}

The Kottwitz homomorphism gives an isomorphism
$$
\kappa_M : \Omega_M^\sigma ~ \widetilde{\rightarrow} ~ X^*(Z(\widehat{M}))^\sigma_I
$$
(cf. \cite{Ko97}, 7.7).  Putting this together with the lemma we get Theorem \ref{Cart_decomp_stmt}.

The proof of Lemma \ref{Cartan_lemma} will occupy the next four sections.

\section{Some ingredients about parahoric subgroups}

\subsection{Parahoric subgroups of $F$-Levi subgroups}

As before, let $A$ denote a maximal $F$-split torus in $G$, let $S \supseteq A$ be a maximal $L$-split torus which is defined over $F$, and let $T = C_G(S)$, a maximal torus of $G$ which is defined over $F$.

Let $A_M$ denote any subtorus of $A$, and let $M = C_G(A_M)$.  Thus $M$ is a semi-standard $F$-Levi subgroup of $G$.  The extended buildings $\mathcal B(M(L))$ and $\mathcal B(G(L))$ share an apartment (which corresponds to $S$), but the affine hyperplanes in the apartment $\mathcal A^M_L$ for $M(L)$ form a subset of those in the apartment $\mathcal A_L$ for $G(L)$.  Hence any facet ${\bf a}_J$ in $\mathcal A_L$ is contained in a unique facet in $\mathcal A^M_L$, which we will denote by ${\bf a}^M_J$.

The following result was proved in \cite{H08} in the special case where $G$ splits over $L$.

\begin{lemma} \label{parahoric_cap_Levi} Suppose $J(L) \subset G(L)$ is the parahoric subgroup corresponding to a facet 
${\bf a}_J \subset \mathcal A_L$.  Then $J(L) \cap M$ is a parahoric subgroup of $M(L)$, and corresponds to the facet ${\bf a}^M_J \subset \mathcal A^M_L$.
\end{lemma}

\begin{proof}
The main result result of \cite{HR} is the following characterization of parahoric subgroups:
$$
J(L) = {\rm Fix}({\bf a}_J)  \cap G(L)_1.
$$
Applying this for the groups $M$ and $G$, we see we only need to show
$$
{\rm Fix}({\bf a}_J) \cap G(L)_1 \cap M(L) = {\rm Fix}({\bf a}^M_J) \cap M(L)_1.
$$
The functoriality of the Kottwitz homomorphisms shows $M(L)_1 \subset G(L)_1$, and then the inclusion "$\supseteq$" is evident.  Let ${\bf a}^M$  denote an alcove in $\mathcal A^M_L$ whose closure contains ${\bf a}^M_J$.  Let $I_M$ denote the Iwahori subgroup of $M(L)$ corresponding to 
${\bf a}^M$. 

Let $S^M_{\rm sc}$ resp. $T^M_{\rm sc}$ denote the pull-back of the torus $(S \cap M_{\rm der})^\circ$ resp. $T \cap M_{\rm der}$ along the homomorphism $M_{\rm sc} \rightarrow M_{\rm der}$.  To prove the inclusion ``$\subseteq$'' it is enough to prove the following claim, since $N_{M_{\rm sc}}(S^M_{\rm sc})(L)$ and $I_M$ belong to $M(L)_1$.  Here and in what follows, we abuse notation slightly by writing $N_{M_{\rm sc}}(S^M_{\rm sc})(L)$ where we really mean its image in $M(L)$.
\smallskip

\noindent {\bf Claim:}  Any element $m \in M(L) \cap G(L)_1$ which fixes a point in ${\bf a}^M_J$ belongs to $$I_M \, N_{M_{\rm sc}}(S^M_{\rm sc})(L) \,I_M$$ and fixes every point of ${\bf a}^M_J$.

\smallskip

\noindent {\em Proof:}  Recall the decomposition
\begin{equation} \label{HR_decomp_I_M}
I_M \backslash M(L) / I_M  \cong N_M(S)(L)/T(L)_1
\end{equation}
of \cite{HR}, Prop.~8.  Using this we may assume $m \in N_M(S)(L)$.

 We will show that for such an element $m$ which fixes a point of ${\bf a}^M_J$ we have $m \in T(L)_1 \, N_{M_{\rm sc}}(S^M_{\rm sc})(L)$, which will prove the first statement of the claim.  It will also prove the second statement, since then $m$ determines a type-preserving automorphism of the apartment $\mathcal A^M_L$, hence fixes ${\bf a}^M_J$ if it fixes any of its points.  

Choose a special vertex ${\bf a}^M_0$ contained in the closure of ${\bf a}^M$, and let $K_0$ denote the corresponding special maximal parahoric subgroup of $M(L)$.  We may write $m = tn$, where $t \in T(L)$ and $n \in N_M(S)(L) \cap K_0$ (cf. \cite{HR}, Prop. 13).  Define $\nu \in X_*(T)_I$ to be $\kappa_T(t)$ and $w \in W(M,S)$ to be the image of $n$ under the projection $N_M(S)(L) \rightarrow W(M,S)$.  Thus $m$ maps to the element $t_\nu \, w \in X_*(T)_I \rtimes W(M,S) \cong \widetilde{W}_M$, the Iwahori-Weyl group for $M$.  

Let $\Sigma^\vee$ denote the coroots associated to the unique reduced root system $\Sigma$ such that the set of affine roots $\Phi_{\rm af}(G(L), S)$ on $\mathcal A_L$ are given by $\Phi_{\rm af} = \{ \alpha + k ~ | ~ \alpha \in \Sigma, \,  k \in \mathbb Z \}$, cf. \cite{HR}.   Let $\Sigma_M^\vee$ denote the coroots for the corresponding root system $\Sigma_M$ for $\Phi_{\rm af}(M(L),S)$ on $\mathcal A^M_L$.  Let $Q^\vee(\Sigma)$ resp. $Q^\vee(\Sigma_M)$ denote the lattice spanned by $\Sigma^\vee$ resp. $\Sigma^\vee_M$.  Recall from \cite{HR} that we have identifications $Q^\vee(\Sigma) \cong X_*(T_{\rm sc})_I$ and $Q^\vee(\Sigma_M) \cong X_*(T^M_{\rm sc})_I$.  Also, we have $\Phi_{\rm af}(M(L),S) \subseteq \Phi_{\rm af}(G(L),S)$, and therefore $Q^\vee(\Sigma_M )\subseteq Q^\vee(\Sigma)$. 

Clearly $w$ is the image of an element from $N_{M_{\rm sc}}(S^M_{\rm sc})(L) \cap K_{0}$, since the latter also surjects onto $W(M,S)$.  Thus we need only show that $\nu \in Q^\vee(\Sigma_M)$, since $Q^\vee(\Sigma_M)$ is also in the image of $N_{M_{\rm sc}}(S^M_{\rm sc})(L) \rightarrow \widetilde{W}_M$.  

First, we will prove that $\nu \in Q^\vee(\Sigma)$.  Indeed, 
by construction $t \in G(L)_1$, and using 
$$X_*(T)_I/X_*(T_{\rm sc})_I \cong X^*(Z(\widehat{G}))_I$$ 
(cf. \cite{HR}) we see that $\nu \in X_*(T_{\rm sc})_I = Q^\vee(\Sigma)$.

Next, let $r$ denote the order of $w \in W(M,S)$.  The element $m^r$ maps to $(t_\nu w)^r \in \widetilde{W}_M$, which is the translation by the element $\mu := \sum_{i=0}^{r-1}w^i\nu \in Q^\vee(\Sigma)$.  But as this translation fixes a point of ${\bf a}^M_J$, we must have $\mu = 0$.  Since $w^i\nu \equiv \nu$ modulo $Q^\vee(\Sigma_M)$, it follows that 
$$
\nu \in Q^\vee(\Sigma_M)_{\mathbb Q} \cap Q^\vee(\Sigma) = Q^\vee(\Sigma_M).
$$
This completes the proof of the claim, and thus the lemma.
\end{proof}

\subsection{Parahoric subgroups of minimal $F$-Levi subgroups}

Now we return to the usual notation, where $M := C_G(A)$ is a minimal $F$-Levi subgroup of $G$.  In this case $M_{\rm ad}$ is anisotropic over $F$ and the semi-simple building $\mathcal B(M_{\rm ad}(F)) = \mathcal B(M_{\rm ad}(L))^\sigma$ is a singleton.  The apartment $(\mathcal A^M_L)^\sigma$ is the empty apartment (no affine hyperplanes).  Therefore, $M(F)$ has only one parahoric subgroup.

\begin{lemma} \label{para_cap_min_Levi}
Let $J$ be any parahoric subgroup of $G(L)$ corresponding to a $\sigma$-invariant facet ${\bf a}_J$ in $\mathcal A_L$.  Then $J(L) \cap M(F) = M(F)_1$.
\end{lemma}

\begin{proof}
By Lemma \ref{parahoric_cap_Levi}, the inclusion "$\subseteq$" is clear.  Let $m \in M(F)_1$.  Since $m$ acts trivially on the apartment $\mathcal A_L^\sigma$ in the building $\mathcal B(G(F)) = \mathcal B(G(L))^\sigma$, it fixes a point of the $\sigma$-invariant facet ${\bf a}_J$ (e.g. its barycenter).  But then since $m \in G(F)_1$, by the Claim in the proof of Lemma \ref{parahoric_cap_Levi} (taking $M =G$), $m$ fixes every point in ${\bf a}_J$.  Clearly then $m \in {\rm Fix}({\bf a}_J) \cap G(L)_1 \cap M(F) = J(L) \cap M(F)$.
\end{proof}

\begin{lemma} \label{inject}
Let $K(L)$ denote the parahoric subgroup of $G(L)$ whose $\sigma$-fixed subgroup $K = K(L)^\sigma$ is the special maximal compact subgroup of $G(F)$ we fixed earlier.  Then
$$
K \cap N_G(S)(L) \cap M(F) = T(F)_1.
$$
\end{lemma}

\begin{proof}
Fix an Iwahori subgroup $I \subset G(L)$ corresponding to a $\sigma$-invariant alcove in $\mathcal A_L$.  Note that by Lemma \ref{para_cap_min_Levi}, we have
$K \cap M(F) = I \cap  M(F)$ and hence
$$
K \cap N_G(S)(L) \cap M(F) = I \cap N_G(S)(L) \cap M(F).
$$
By \cite{HR}, Lemma 6, the right hand side is $T(L)_1 \cap M(F) = T(F)_1$.  
\end{proof}

\section{The isomorphism $\widetilde{W}_K^\sigma \cong W(G,A)$}

By \cite{HR}, Remark 9, any element of $\widetilde{W}_K^\sigma$ is represented by an element of $N_G(S)(F)$.  Let $x \in N_G(S)(F)$.  Then $xSx^{-1} = S$ contains $xAx^{-1}$ and $A$, which being maximal $F$-split tori in $S$, must coincide.  Thus, there is a tautological homomorphism
$$
N_G(S)(F) \rightarrow N_G(A)(F).
$$
By Lemma \ref{inject}, this factors to give an injective homomorphism
$$
\widetilde{W}_K^\sigma \hookrightarrow W(G,A).
$$
The next statement furnishes the proof of Lemma \ref{Cartan_lemma}, (I).

\begin{lemma} \label{taut_hom}
The homomorphism $\widetilde{W}_K^\sigma \rightarrow W(G,A)$ is an isomorphism.  This allows us to regard $W(G,A)$ as a subgroup of $\widetilde{W}^\sigma$.
\end{lemma}

\begin{proof}
It is enough to prove the domain and codomain have the same order.  Let $k_L$ denote the residue field of $\mathcal O_L$, which can be identified with an algebraic closure of $k_F$.  Consider the special fiber $\overline{\mathcal G}^\circ_{\bf a_0} = {\mathcal G}^\circ_{{\bf a}_0} \times_{\mathcal O_L} k_L$ of the Bruhat-Tits group scheme ${\mathcal G}^\circ_{{\bf a}_0}$ over $\mathcal O_L$ which is associated to the facet ${\bf a}_0$ in the building $\mathcal B(G(L))$.  Let $\overline{\mathcal G}^{\circ, \rm red}_{{\bf a}_0}$ denote the maximal reductive quotient of $\overline{\mathcal G}^\circ_{{\bf a}_0}$.  By \cite{HR}, Prop. 12, $\widetilde{W}_K$ is the Weyl group of $\overline{\mathcal G}^{\circ, \rm red}_{\bf a_0}$.  The group $\overline{\mathcal G}^{\circ, \rm red}_{{\bf a}_0}$ is defined over $k_F$, and in fact we have 
$\overline{\mathcal G}^{\circ,\rm red}_{{\bf a}_0}  = 
\overline{\mathcal G}^{\circ, \rm red}_{v_F} 
\times_{k_F} k_L$, where $\overline{\mathcal G}^\circ_{v_F}$ is the special fiber of $\mathcal G^\circ_{v_F}$ (cf. \cite{Land}, Cor. 10.10).  Since $k_F$ is finite, $\overline{\mathcal G}^{\circ, \rm red}_{v_F}$ is automatically quasi-split over $k_F$, and it follows that $\widetilde{W}_K^\sigma$ is the Weyl group of $\overline{\mathcal G}^{\circ, \rm red}_{v_F}$ (this is well-known, but one can also use the argument which yields Remark \ref{qs_Weyl_bijection} below).  

On the other hand, by \cite{Tits}, 3.5.1, the root system of $\overline{\mathcal G}^{\circ, \rm red}_{v_F}$ is $\Phi_{v_F}$, the root system consisting of the vector parts of the affine roots for $A$ which vanish on $v_F$ (loc.~cit. 1.9).  Because $v_F$ is special, $\Phi_{v_F} = \Phi(G,A)$, the relative root system.  Thus the Weyl group of $\overline{\mathcal G}^{\circ, \rm red}_{v_F}$ is isomorphic to $W(G,A)$. 

These remarks imply that $\widetilde{W}_K^\sigma$ and $W(G,A)$ are abstractly isomorphic groups and in particular they have the same order.  

\end{proof}

\section{A decomposition of the Iwahori Weyl group}

The goal here is to prove Lemma \ref{Cartan_lemma}, (II).

\subsection{A lemma on finite Weyl groups}

Let $w \in W(G,A)$ and choose a representative $g \in N_G(A)(F)$ for $w$; write $[g] = w$.  The tori $gSg^{-1}$ and $S$ are both maximal $L$-split tori in $M$, hence there exists $m \in M(L)$ such that $mgSg^{-1}m^{-1} = S$.  We claim that the map
\begin{align*}
W(G,A) &\rightarrow W(G,S)/W(M,S) \\
w &\mapsto [mg]\cdot W(M,S)
\end{align*}
is well-defined and injective.  Indeed, suppose $g_0 \in N_G(A)(F)$ represents an element $w_0 \in W(G,A)$ and that $m_0 \in M(L)$ satisfies $m_0g_0 S g_0^{-1}m_0^{-1} = S$.  To show the map is well-defined, we suppose $w= w_0$ and we show that $(mg)^{-1}m_0g_0 \in N_M(S)$.  It will suffice to show $(mg)^{-1}m_0g_0$ belongs to $M(L)$.  Since $g$ normalizes $M = C_G(A)$ and $g^{-1}g_0 \in M$, this is obvious.   To show the map is injective we suppose $[mg]W(M,S) =  [m_0g_0] W(M,S)$, that is, $(mg)^{-1}m_0g_0 \in N_M(S)$.  Arguing as before, we deduce that $g^{-1}g_0 \in M$.  This shows that $w = w_0$ and so we get the injectivity.

\begin{Remark} \label{alt_desc} Here is another way to describe the map.  For an element $w \in W(G,A)$, using Lemma \ref{taut_hom} choose an element $x \in N_G(S)(F) \cap K$ whose image in $\widetilde{W}^\sigma_K$ maps to $w$ under the isomorphism $\widetilde{W}_K^\sigma ~ \widetilde{\rightarrow}  ~ W(G,A)$.  Then the map sends $w$ to the coset $[x] W(M,S)$. 
\end{Remark}

\begin{lemma} \label{Weyl_bijection}
The above map induces a bijection
$$
W(G,A) ~ \widetilde{\rightarrow} ~ [W(G,S)/W(M,S)]^\sigma.
$$
\end{lemma}

\begin{proof}
First we prove the image $[mg]W(M,S)$ is $\sigma$-invariant.  This follows because the element  $(mg)^{-1} \sigma(m)g$ belongs to $M$, hence to $N_M(S)$.

Next we prove the surjectivity.  Suppose $x \in N_G(S)$ projects to an element in $W(G,S)$ which represents a $\sigma$-fixed coset $C$ in $W(G,S)/W(M,S)$, that is, $x^{-1}\sigma(x) \in M$.  Then the subtorus $xAx^{-1} \subset S$ is defined over $F$.   The inner automorphism ${\rm Int}(x): S \rightarrow S$, restricted to $A$ gives an isomorphism ${\rm Int}(x): A ~ \widetilde{\rightarrow}~ xAx^{-1}$ which is defined over $F$.  It follows that $xAx^{-1}$ is $F$-split.  Since $A$ and $xAx^{-1}$ are maximal $F$-split tori in $S$, they coincide.  Thus $x \in N_G(A)$, and the image of $x$ is the coset $C$.
\end{proof}

\begin{Remark} \label{qs_Weyl_bijection}
If $G$ is quasi-split over $F$, then $M = T $ and we recover the well-known result that $W(G,A) = 
W(G,S)^\sigma$.
\end{Remark}

\subsection{Proof of the decomposition}

We keep the notation of the previous subsection.  There is a commutative diagram of exact sequences with $\sigma$-equivariant morphisms and injective vertical maps
$$
\xymatrix{
0 \ar[r] & X_*(T)_I \ar[r] \ar[d]^{=} & \widetilde{W}_M \ar[r] \ar[d] & W(M,S) \ar[r] \ar[d] & 0 \\
0 \ar[r] & X_*(T)_I \ar[r]  & \widetilde{W} \ar[r]  & W(G,S) \ar[r]  & 0}
$$
(see \cite{HR}, Prop. 13).   The canonical map $\widetilde{W}_M \backslash \widetilde{W}\rightarrow W(M,S) \backslash W(G,S)$ is bijective and $\sigma$-equivariant, so we get
$$
[\widetilde{W}_M \backslash \widetilde{W}]^\sigma \cong [W(M,S) \backslash W(G,S)]^\sigma.
$$
Using the map $W(G,A) \hookrightarrow \widetilde{W}^\sigma$ constructed in Lemma \ref{taut_hom} we get a commutative diagram
$$
\xymatrix{
W(G,A) \ar[r] \ar[dr] & \widetilde{W}_M^\sigma \backslash \widetilde{W}^\sigma \ar[d] \\
 & (\widetilde{W}_M \backslash \widetilde{W})^\sigma.}
$$
The commutativity of this diagram follows using Remark \ref{alt_desc}.  Since the diagonal arrow is a bijection by the above discussion, and the vertical arrow is obviously an injection, it follows that all arrows in the diagram are bijections.  The decomposition
$$
\widetilde{W}^\sigma = \widetilde{W}^\sigma_M \cdot W(G,A)
$$
follows.  It is clear that $W(G,A)$ normalizes $\widetilde{W}^\sigma_M$.  This completes the proof of Lemma \ref{Cartan_lemma},(II) .

\section{End of proof of the Cartan decomposition}

\subsection{Invariants in the affine Weyl group of $M$}

\begin{lemma} \label{W_M_aff_sigma} Let $M$ again denote a minimal $F$-Levi subgroup, and let $W_{M,\rm aff}$ denote the affine Weyl group associated to $M$.  Then $W_{M, \rm aff}^\sigma = 1$.  
\end{lemma}

\begin{proof}
We identify $W_{M,\rm aff}$ with the Iwahori-Weyl group $N_{M_{\rm sc}}(S^M_{\rm sc})(L)/T^M_{\rm sc}(L)_1$.  Let $I_{M_{\rm sc}}$ denote the Iwahori subgroup of $M_{\rm sc}(L)$ corresponding to a $\sigma$-invariant alcove ${\bf a}^{M_{\rm sc}}$ in the apartment $\mathcal A^{M_{\rm sc}}_L = X_*(S^M_{\rm sc})_{\mathbb R}$ of $\mathcal B(M_{\rm sc}(L))$ associated to the torus $S^M_{\rm sc}$.  By \cite{HR}, Remark 9, the set $W_{M,\rm aff}^\sigma$ is in bijective correspondence with
$$
I_{M_{\rm sc}}(F) \backslash M_{\rm sc}(F) / I_{M_{\rm sc}}(F).
$$
Therefore it is enough to prove that $M_{\rm sc}(F) = I_{M_{\rm sc}}(F)$.  
But $M_{\rm sc}(F) = M_{\rm sc}(F)_1 \subseteq I_{M_{\rm sc}}$.  To prove the inclusion, note that an element in $M_{\rm sc}(F)_1$ acts trivially on the apartment $\mathcal A^{M_{\rm sc}}_L$ (cf. the Claim above), hence fixes ${\bf a}^{M_{\rm sc}}$.   Thus $M_{\rm sc}(F) = I_{M_{\rm sc}}(F)$ and we are done.
\end{proof}

\subsection{Conclusion of the proof of Theorem \ref{Cart_decomp_stmt}}

We have fixed the $\sigma$-stable alcove ${\bf a}$ and this determines the $\sigma$-stable alcove ${\bf a}^M$ and the corresponding subgroup $\Omega_M \subset \widetilde{W}_M$.  
There is a canonical $\sigma$-equivariant decomposition $\widetilde{W}_M = W_{M,\rm aff} \rtimes \Omega_M$, so in view of the above lemma, we deduce that
$$
\widetilde{W}_M^\sigma = \Omega_M^\sigma.
$$
This completes the proof of the last part, namely (III), of Lemma \ref{Cartan_lemma}.  Since the Theorem \ref{Cart_decomp_stmt} is a consequence of Lemma \ref{Cartan_lemma}, we have proved Theorem \ref{Cart_decomp_stmt}. \qed

\section{Characterization of special maximal compact subgroups} \label{Ktilde_char_sec}

Let 
$$v_G : G(L) \rightarrow X^*(Z(\widehat{G}))_I/torsion$$ denote the homomorphism derived from the Kottwitz homomorphism $$\kappa_G: G(L) \rightarrow X^*(Z(\widehat{G}))_I$$ in the obvious way.  Denote its kernel by $G(L)^1$ and let $G(F)^1 = G(L)^1 \cap G(F)$.  Note that if $M$ is a minimal $F$-Levi subgroup of $G$, then $M(F)^1$ is the unique maximal compact open subgroup of $M(F)$, consistent with the notation used in the introduction.

Let $K := \mathcal G_{v_F}^\circ(\calO_F)$, the maximal parahoric subgroup of $G(F)$ corresponding to 
$v_F$.  By \cite{HR}, Prop.~3 and Remark 9, we have the equality
\begin{equation*}
K = G(F)_1 \cap {\rm Fix}({\bf a}_0).
\end{equation*}

Using the Claim from the proof of Lemma \ref{parahoric_cap_Levi} in the case $M = G$, we derive the equality
\begin{equation} \label{K=}
K = G(F)_1 \cap {\rm Fix}(v_F).
\end{equation}

Our goal is to prove the analogous description of $\widetilde{K}$.

\begin{lemma} \label{tildeK_char} The special maximal compact subgroups of $G(F)$ are precisely the subgroups of the form
\begin{equation} \label{tildeK=}
\widetilde K = G(F)^1 \cap {\rm Fix}(v_F),
\end{equation}
where $v_F$ ranges over the special vertices in the building $\mathcal B(G_{\rm ad}(F))$.
\end{lemma}

\begin{proof} A compact subgroup of $G(F)$ is automatically contained in $G(F)^1$.  This follows from the alternative description of $G(L)^1$ as the intersection of the kernels of the homomorphisms $|\chi|:G(L) \rightarrow \mathbb R_{>0}$, where $\chi$ ranges over $L$-rational characters on $G$. 

Thus, using \cite{BT1}, Cor.~(4.4.1), every maximal compact subgroup $\widetilde{K}$ of $G(F)$ (equiv., of 
$G(F)^1$) is the stabilizer in $G(F)^1$ of a well-defined facet in the building $\mathcal B(G_{\rm der}(F))$.  By definition, such a $\widetilde{K}$ is special if and only if the facet it stabilizes is a special vertex $v_F$.   In that case, we have $\widetilde{K} = G(F)^1 \cap {\rm Fix}(v_F)$.  

To show the converse, we must check that $G(F)^1 \cap {\rm Fix}(v_F)$ is compact (the argument above will then show it is (special) maximal compact).  Recall $K = \mathcal G_{v_F}^\circ(\calO_F)$ is compact and is given by (\ref{K=}).    Since $G(F)_1 \cap {\rm Fix}(v_F)$ has finite index in $G(F)^1 \cap {\rm Fix}(v_F)$, and since the former is compact, so is the latter.  This completes the proof.
\end{proof}

\begin{Remark}
Equation (8.0.1) can be generalized.  Let ${\bf a}_J$ denote any $\sigma$-stable alcove in $\mathcal B(G(L))$.  Then
$$
\mathcal G^\circ_{{\bf a}_J}(\calO_F) = G(F)_1 \cap {\rm Fix}({\bf a}^\sigma_J).
$$
\end{Remark}

\section{Statement of the Satake isomorphism}

In this section, let $P = MN$ denote any $F$-rational parabolic subgroup of $G$ with unipotent radical $N$, which has $M$ as a Levi factor.

\subsection{Iwasawa decomposition}

In light of Lemma \ref{tildeK_char}, the following version of the Iwasawa decomposition can be derived easily from similar statements in the literature (cf. \cite{BT1}, Rem.~(4.4.5) or Prop.~(7.3.1)):

\begin{prop} 
There is an equality of sets $$G(F) = P(F) \cdot \widetilde{K}(F).$$
\end{prop}

We need the variant of this where $\widetilde{K}(F)$ is replaced by $K(F)$.  It will be enough to prove that
$$
\widetilde{K}(F) = (\widetilde{K} \cap M(F)) \cdot K(F).
$$
Using (\ref{HR_decomp}) 
together with Lemma \ref{Cartan_lemma}, we see that any element $\tilde k \in \widetilde{K}(F)$ satisfies
$$
\tilde k \in K(F) m K(F)
$$
for some $m \in \Omega_M^\sigma \subset M(F)$.  It follows that $m \in \widetilde{K}(F)$, and then since $\widetilde{K}(F)$ normalizes $K(F)$ (cf.~e.g.~Lemma \ref{tildeK_char}), we see that $\tilde k \in m K(F)$ as desired.

We have thus proved the first part of the following corollary.

\begin{cor} [Iwasawa decomposition] \label{Iwasawa}
There is an equality of sets 
$$G(F) = P(F) \cdot K(F).$$
Moreover, $P(F) \cap K(F) = (M(F) \cap K) \cdot (N(F) \cap K)$.
\end{cor}

\begin{proof}
We need only show the second equality, which can be rewritten as
$$
P(F) \cap \mathcal G^\circ_{v_F}(\mathcal O_F) = (M(F) \cap \mathcal G^\circ_{v_F}(\mathcal O_F)) \cdot (N(F) \cap \mathcal G^\circ_{v_F}(\mathcal O_F)). 
$$
This follows from \cite{BT2}, 5.2.4 (taking the set denoted by $\Omega$ there to be $\{ v_F \}$).
\end{proof}

\subsection{Construction of the Satake transform}

We will follow the approach taken in \cite{HKP}, which treated the case of $F$-split groups. 

Recall that $\calH_K := C_c(K(F) \backslash G(F)/K(F))$, the spherical Hecke algebra of $K(F)$-bi-invariant compactly-supported functions on $G(F)$.  The convolution is defined using the Haar measure on $G(F)$ which gives $K(F)$ volume 1.

Set $R := \mathbb C[M(F)/M(F)_1]$.  Since $M(F)_1$ is the unique parahoric subgroup of $M(F)$, this is just the Iwahori-Hecke algebra for $M(F)$.  Let ${\bf  M} := C_c(M(F)_1 N(F) \backslash G(F)/ K(F))$, where the subscript ``c'' means we consider functions supported on finitely many double cosets.  Then ${\bf M}$ carries an obvious right convolution action under $\calH_K$.  It also carries a left action by $R$ given by normalized convolutions:
$$
r \cdot \phi(m) := \int_{M(F)} \delta_P^{1/2}(m_1) \, r(m_1) \, \phi(m_1^{-1} m) \, dm_1.
$$
Here  $dm_1$ is the Haar measure on $M(F)$ giving $M(F)_1$ volume 1, and $\delta_P$ is the modular function on $P(F)$ given by the normalized absolute value of the determinant of the adjoint action on ${\rm Lie}(N(F))$.  For $m \in M(F)$ we have
$$
\delta_P(m) := |{\rm det}({\rm Ad}(m) ~ ; ~ {\rm Lie}(N(F)))|_F.
$$
The actions of $R$ and $\calH_K$ on ${\bf M}$ commute, so that ${\bf M}$ is an $(R,\calH_K)$-bimodule.

\begin{lemma}
The $R$-module ${\bf M}$ is free of rank 1, with canonical generator 
$$v_1 := {\rm char}(M(F)_1 \, N(F) \, K(F)).$$
\end{lemma} 

\begin{proof}
This follows directly from Proposition \ref{Iwasawa}.
\end{proof}

Given $f \in \calH_K$, let $f^\vee \in R$ denote the unique element satisfying the identity
\begin{equation} \label{transform_def}
v_1 f = f^\vee v_1.
\end{equation}

It is obvious that 
\begin{align*}
\calH_K & \rightarrow R \\
f &\mapsto f^\vee 
\end{align*}
is a $\mathbb C$-algebra homomorphism.
 
Evaluating both sides of (\ref{transform_def}) on $m \in M(F)$ and using the usual $G = MNK$ integration formula (see \cite{Car}), we get the familiar expression
\begin{equation} 
f^\vee(m) = \delta_P^{-1/2}(m) \, \int_{N(F)} f(nm) \, dn = \delta_P^{1/2}(m) \, \int_{N(F)} f(mn) \, dn,
\end{equation}
where $dn$ gives $N(F) \cap K(F)$ measure 1.

\section{The Satake transform is an isomorphism}

\subsection{Weyl group invariance}
The first step is to prove that $f^\vee$ belongs to the subring $R^{W(G,A)}$ of $W(G,A)$-invariants in $R$.  Once this is proved, the functoriality of the Kottwitz homomorphism
$$
\kappa_M: M(F)/M(F)_1 ~ \widetilde{\rightarrow} ~ X^\ast(Z(\widehat{M}))^\sigma_I
$$ 
shows that $f^\vee \in \mathbb C[X^\ast(Z(\widehat{M}))^\sigma_I]^{W(G,A)}$, as well. 

The argument is virtually the same as Cartier's \cite{Car}.  Define a function on $m \in M(F)$ by
$$
D(m) = |{\rm det}({\rm Ad}(m) - 1 \, ; \, {\rm Lie}~G(F) /{\rm Lie}~M(F))|^{1/2}.
$$
Then exactly as in loc.~cit. one can prove the formula
\begin{equation} \label{fvee=orb.int}
f^\vee(m) = D(m) \int_{G/A} f(gmg^{-1}) \, \dfrac{dg}{da}
\end{equation}
on the Zariski-dense subset of elements $m \in M(F)$ which are regular semi-simple as elements in $G$.  Here $dg$ (resp. $da$) is the Haar measure on $G(F$) (resp. $A(F)$) which gives $K$ (resp. $K \cap A(F)$) volume 1.  By Lemma \ref{Cartan_lemma} (I), every element $w \in W(G,A)$ can be represented by an $x \in N_G(A)\cap K$.  Clearly $D(m) = D(xmx^{-1})$. Since the measure on $G/A$ is invariant under conjugation by $x$, we see as in loc.~cit.~that the integral in (\ref{fvee=orb.int}) is also invariant under $m \mapsto xmx^{-1}$.  Thus (\ref{fvee=orb.int}) is similarly invariant, as desired.

\begin{Remark} As in the case of $\calH_{\widetilde K}$, equation (\ref{fvee=orb.int}) also shows that $f^\vee$ is independent of the choice of $F$-rational parabolic subgroup $P$ which contains $M$ as a Levi factor.
\end{Remark}

\subsection{Upper triangularity}
The second step is to show that with respect to natural $\mathbb C$-bases of $\calH_K$ and $R^{W(G,A)}$, the map $f \mapsto f^\vee$ is ``invertible upper triangular'', hence is an isomorphism of algebras.

The set $\widetilde{W}^\sigma_K \backslash \widetilde{W}^\sigma/\widetilde{W}^\sigma_K \cong W(G,A)\backslash \Omega_M^\sigma$ provides a natural $\mathbb C$-basis for $\calH_K$ and for $R^{W(G,A)}$.  Recall that $\widetilde{W}$ has a natural structure of a {\em quasi-Coxeter group}
$$
\widetilde{W} = W_{\rm aff} \rtimes \Omega
$$
(cf. \cite{HR}, Lemma 14).  We extend the Bruhat order $\leq$ and the length function $\ell$ from $W_{\rm aff}$ to $\widetilde{W}$ in the usual way (cf.~loc.~cit.).  Given $x \in \widetilde{W}$, denote by $\tilde x \in \widetilde{W}$ the unique minimal element in $\widetilde{W}_K x \widetilde{W}_K$.  (Note that $\widetilde{W}_K$ is finite and that the usual theory of such minimal elements for Coxeter groups goes over to handle quasi-Coxeter groups.)  

By \cite{HR}, Remark 9, we may regard $\widetilde{W}_K^\sigma \backslash \widetilde{W}^\sigma/ \widetilde{W}_K^\sigma$ as a subset (the $\sigma$-invariant elements) in $\widetilde{W}_K \backslash \widetilde{W} / \widetilde{W}_K$.  For $y,y' \in W(G,A)\backslash \Omega_M^\sigma$ resp. $x,x' \in \widetilde{W}_K^\sigma \backslash \widetilde{W}^\sigma/ \widetilde{W}_K^\sigma$, we define the partial order $\preceq$ by requiring
\begin{align*}
y \preceq y' &\Leftrightarrow \tilde y \leq \tilde y', \,\,\,\, \mbox{resp.} \\
x \preceq x' &\Leftrightarrow \tilde x\leq \tilde x'.
\end{align*}
The set $W(G,A)\backslash \Omega_M^\sigma$ is countable and every element $y$ has only finitely many predecessors with respect to the partial order $\preceq$.   Therefore there is a total ordering $y_1, y_2, \dots$ on this set which is compatible with $\preceq$, meaning that $y_i \preceq y_j$ only if $i \leq j$.  Similar remarks apply to the partially ordered set $\widetilde{W}_K^\sigma \backslash \widetilde{W}^\sigma/ \widetilde{W}_K^\sigma$, and we get an analogous total ordering $x_1, x_2, \dots$ for it. 

We claim that the matrix for $f \mapsto f^\vee$ in terms of the bases $\{ y_i \}_1^\infty$ and $\{x_i\}_1^\infty$ is upper triangular and invertible.  The upper triangularity is the content of the next lemma.

\begin{lemma} \label{upper_triang_lem}
Suppose $x \in \widetilde{W}^\sigma$ and $y \in \Omega_M^\sigma$ and that 
\begin{equation} \label{NcapK}
N(F)y K(F) \cap K(F) x K(F) \neq \emptyset.  
\end{equation}
Then $\tilde y \leq \tilde x$.
\end{lemma}

\begin{proof}
Let $I$ denote the Iwahori subgroup of $G(L)$ associated to the $\sigma$-stable alcove ${\bf a}$, as defined earlier.  We shall need two BN-pair relations.  The first is the relation
\begin{equation} \label{1st_rel}
K(L) = I(L)\, \widetilde{W}_K\, I(L).
\end{equation} 
This follows easily using \cite{HR}, Prop.~8.   The second is the relation
\begin{equation} \label{2nd_rel}
I(L)\, w \,I(L)\, w'\, I(L) \subseteq \coprod_{w'' \leq w'} I(L)\, w\, w''\, I(L).
\end{equation}
This relation per se does not appear in the literature, but it follows easily from the BN-pair relations established in \cite{BT2}, 5.2.12 (cf. \cite{HR}, paragraph following Lemma 17).

Using (\ref{1st_rel}) and (\ref{2nd_rel}) we see that (\ref{NcapK}) implies that 
\begin{equation} \label{NcapI}
N(L) y I(L) \cap I(L) \, x' \, I(L) \neq \emptyset
\end{equation}
for some $x' \in \widetilde{W}_K x \widetilde{W}_K$.  Write
\begin{equation} \label{ny}
ny = i \, x' \, i'
\end{equation}
for $n \in N(L)$, and $i, i' \in I(L)$.  Choose a cocharacter $\lambda \in X_*(A)$ such that $\varpi^\lambda n \varpi^{-\lambda} \in I(L)$.  Then multiplying (\ref{ny}) by $\varpi^\lambda$ we see that
$$
I(L) \,\varpi^\lambda y \, I(L) \subseteq I(L) \, \varpi^\lambda \, I(L) \, x' \, I(L).
$$
Using (\ref{2nd_rel}) again we deduce that 
$$
I(L) \, \varpi^\lambda y \, I(L) = I(L) \, \varpi^\lambda x'' \, I(L)
$$
and hence $y = x''$ for some $x'' \in \widetilde{W}$ with $x'' \leq x'$.  Thus $\tilde{y} \leq x'$.  A standard argument then shows that $\tilde y \leq \tilde{x}$, which is what we wanted to prove. 
\end{proof}

Finally, the invertibility follows from the obvious fact that
$$
N(F)x K(F) \cap K(F) x K(F) \neq \emptyset.
$$

This completes the proof that $f \mapsto f^\vee$ is an isomorphism. \qed

\section{The structure of $\Lambda_M$} \label{structure_sec}

It is clear that $\Lambda_M = X^\ast(Z(\widehat{M}))_I^\sigma$ is a finitely-generated abelian group.  In this section we make it more concrete in various situations.

\subsection{General results}

As before, in this subsection $T$ denotes the centralizer in $G$ of the torus $S$.  Recall that we can assume $S$ is defined over $F$, and so $T$ is also defined over $F$.  Recall also that $T^M_{\rm sc}$ denotes the pull-back of $T$ via $M_{\rm sc} \rightarrow M$.

\begin{lemma} \label{extn}
There is an embedding $X_*(T)^\sigma_I \hookrightarrow \Lambda_M$ whose cokernel is isomorphic to the finite abelian group ${\rm ker}[X_*(T^M_{\rm sc})_\Gamma \rightarrow X_*(T)_\Gamma]$.   
\end{lemma}

\begin{proof}
Use the long exact sequence for $H^i(\langle \sigma \rangle, -)$ associated to the short exact sequence
$$
\xymatrix{
0 \ar[r] & X_\ast(T^M_{\rm sc})_I \ar[r] & X_\ast(T)_I \ar[r] & X^\ast(Z(\widehat{M}))_I \ar[r] & 0.}
$$
(For a discussion of this short exact sequence, see \cite{HR}, proof of Prop. 13.)
Note that $X_\ast(T^M_{\rm sc})_I^\sigma \subset W_{M,\rm aff}^\sigma = 1$ (cf. Lemma \ref{W_M_aff_sigma}).  Also, $X_*(T^M_{\rm sc})_\Gamma$ is finite because $M_{\rm sc}$ is anisotropic over $F$.  The lemma follows easily using this remarks.  
\end{proof}

\begin{cor} \label{gen_description}
\begin{enumerate}
\item [(a)] If $G$ is quasi-split over $F$, then $\Lambda_M = X_*(T)^\sigma_I$.
\item [(b)] If $G$ is split over $L$, then $\Lambda_M$ fits into the exact sequence
$$
1 \rightarrow X_*(A) \rightarrow \Lambda_M \rightarrow {\rm ker}[X_*(T^M_{\rm sc})_\sigma \rightarrow 
X_*(T)_\sigma] \rightarrow 0.
$$
\item [(c)] If $G$ is unramified over $F$, then $\Lambda_M = X_*(A)$.
\end{enumerate}

\end{cor}

\begin{proof}
Part (a).  Since $G$ is quasi-split over $F$, we have $M = T$, and the desired formula follows directly from the definition of $\Lambda_M$.

Part (b) follows immediately from Lemma \ref{extn}.  

Part (c) follows as a special case of either (a) or (b).  Part (c) was known previously (cf. \cite{Bo}, 9.5).
\end{proof}

\begin{Remark}
If $G$ is semi-simple and anisotropic, then $\Lambda_M$ is finite.  There are examples, namely $G = D^\times/F^\times$ for $D$ a central simple division algebra over $F$ with ${\rm dim}_F(D) > 1$, 
where $\Lambda_M \neq 0$.  

At the opposite extreme, let $E/F$ denote a finite totally ramified extension.  Consider the ``diagonal'' embedding ${\mathbb G}_m \hookrightarrow {\rm R}_{E/F}{\mathbb G}_m$ and set $G = ({\rm R}_{E/F}{\mathbb G}_m)/{\mathbb G}_m$.  Then $\Lambda_G$ is torsion, and non-zero if $E \neq F$.
\end{Remark}

The next proposition tells us how to measure the difference between the subgroups $K$ and $\widetilde{K}$ of $G(F)$ attached to a special vertex $v_F$.  This will complete the proof of Theorem \ref{Sat_thm}.  For an abelian group $H$ let $H_{\rm tor}$ denote its torsion subgroup.

\begin{prop}
There is a set-theoretic inclusion $\Omega_{M, \rm tor}^\sigma \subset \widetilde{K}$ which induces an isomorphism of groups
$$
\Lambda_{M,\rm tor} ~ \widetilde{\rightarrow} ~ \widetilde{K}/K.
$$
\end{prop}

\begin{proof}
Clearly $\Omega_{M, \rm tor}^\sigma$ lies in $M(F)^1$ hence in $G(F)^1$.  Also, every element of $M(F)^1$ acts trivially on the apartment $\mathcal A_L^\sigma$, and in particular, fixes ${\bf a}^\sigma_0$.  This shows that $\Omega_{M, \rm tor}^\sigma \subset {\rm Fix}^{G(F)}(v_F) \cap G(F)^1 = \widetilde{K}$ (cf. Lemma \ref{tildeK_char}).  

We claim the induced homomorphism $\Omega_{M, \rm tor}^\sigma ~ \rightarrow ~ \widetilde{K}/K$ is an isomorphism.  It is injective because
 $$\Omega_M \cap K = \Omega_M \cap M(F) \cap K = \Omega_M \cap M(F)_1 = \{ 1 \}$$
(cf. Lemma \ref{para_cap_min_Levi}).

Let us prove surjectivity.  Any coset in $\widetilde{K}/K$ can be represented by an element $x \in \Omega_M^\sigma$.  We need to show this element is torsion.  Let $r$ be such that $x^r \in K$.  But then $x^r \in \Omega_M^\sigma \cap K = \{ 1 \}$ (see above), and we are done.   
\end{proof}

\begin{cor}
If $M_L$ is $L$-split group and $M_{\rm der} = M_{\rm sc}$, then $\Lambda_M$ is torsion-free, and for every special vertex $v_F$, we have $\widetilde{K}_{v_F} = K_{v_F}$.  
\end{cor}

\begin{proof}
We have
\begin{equation} \label{Lambda_M_I=}
X^*(Z(\widehat{M}))_I = X^*(Z(\widehat{M})) 
\end{equation}
and the latter is torsion free since $M_{\rm der} = M_{\rm sc}$ is equivalent to $Z(\widehat{M})$ being connected.
\end{proof}

\begin{Remark}
The hypotheses on $M$ hold if $G_{\rm der} = G_{\rm sc}$ and $G_L$ is an $L$-split group.  \end{Remark}

\begin{cor}
If $G = G_{\rm sc}$, then $\widetilde{K} = K$ and $\Lambda_M$ is torsion-free.
\end{cor}

\begin{proof}
Observe that since $Z(\widehat{G}) =1$ we have $G(F)_1 = G(F)^1 = G(F)$.  Then use (\ref{K=}) and (\ref{tildeK=}). 
\end{proof}

Of course, this corollary was already known (cf. \cite{BT2}, 4.6.32).

\subsection{Passing to inner forms}  \label{passing_subsec}

It is of interest to describe $\Lambda_M$ explicitly in terms of an appropriate maximal torus $\widehat{T}$ in $\widehat{G}$.  For quasi-split groups this has been done in Corollary \ref{gen_description}, (a), which proves that $\Lambda_M = X^*(\widehat{T})^\sigma_I= X^*(\widehat{T}^I)^\sigma$.  Here we study the effect of passing to an inner form of a quasi-split group.

Thus, we fix a connected reductive group $G^*$ which is quasi-split over $F$.  Recall that an inner form of $G^*$ is a pair $(G,\Psi)$ consisting of a connected reductive $F$-group $G$ and a $\Gamma$-stable $G^*_{\rm ad}(F^s)$-orbit $\Psi$ of $F^s$-isomorphisms $\psi: G \rightarrow G^*$.   The set of isomorphism classes of inner forms of $G^*$ corresponds bijectively to the set $H^1(F,G^*_{\rm ad})$, by the rule which sends $(G,\Psi)$ to the 1-cocycle $\tau \mapsto \psi \circ \tau(\psi)^{-1}$ for any $\psi \in \Psi$ (cf. \cite{Ko97}, 5.2).  

Now assume $(G,\Psi)$ is an inner form of $G^*$.  Denote the action of $\tau \in \Gamma$ on $G(F^s)$ (resp. $G^*(F^s)$) by $\tau$ (resp. $\tau^*$).  

Let $A$ be a maximal $F$-split torus in $G$, and let $S$ denote a maximal $F^{\rm un}$-split torus in $G$ which is defined over $F$ and contains $A$.  Such a torus $S$ exists by \cite{BT2}, 5.1.12, noting that that any $F$-torus which is split over $L$ is already split over $F^{\rm un}$.  Let $T = C_G(S)$ and $M = C_G(A)$.  
Then $T$ is a maximal torus of $G$, since the group $G_{F^{\rm un}}$ is quasi-split.  Let $A^*$, $S^*$, $T^*$ have the corresponding meaning for the group $G^*$, and assume that $T^*$ is contained in an $F$-rational Borel subgroup $B^* = T^* U^*$ of $G^*$.  Of course $T^* = C_{G^*}(A^*)$ since $G^*$ is quasi-split over $F$.

Let $P = MN$ be an $F$-rational parabolic subgroup of $G$ having Levi factor $M$ and unipotent radical $N$.  Let $P^*$ be the unique standard $F$-rational parabolic subgroup of $G^*$ which is $G^*(F^s)$-conjugate to $\psi(P)$ for all $\psi \in \Psi$ (cf. \cite{Bo}, section 3).  Let $M^*$ denote the unique Levi factor of $P^*$ which contains $T^*$.   Let $\Psi_M$ denote the set of $\psi \in \Psi$ such that $\psi(P) = P^*$ and $\psi(M) = M^*$.   Then $\Psi_M$ is a non-empty $\Gamma$-stable $M^*_{\rm ad}(F^s)$-orbit of $F^s$-isomorphisms $M \rightarrow M^*$; hence $M$ is an inner form of the $F$-quasi-split group $M^*$.

It is clear that $G_{F^{\rm un}}$ and $G^\ast_{F^{\rm un}}$ are isomorphic, since they are inner forms of each other and are both quasi-split (cf. \cite{Tits}, 1.10.3).  In fact it is easy to see that any inner twisting $G_{F^{\rm un}} ~ \widetilde{\rightarrow} ~ G^*_{F^{\rm un}}$ over $F^{\rm un}$ is $G^*(F^s)$-conjugate to an {\em isomorphism} of $F^{\rm un}$-groups. For this a key fact is that the image $T^*_{\rm ad}$ of $T^*$ in $G^*_{{\rm ad},F^{\rm un}}$ is an induced $F^{\rm un}$-torus.  The same remarks obviously apply to $M_{F^{\rm un}}$ and $M^*_{F^{\rm un}}$.  Hence we may choose $\psi_0 \in \Psi_M$ such that $\psi_0: M \rightarrow M^*$ is an $F^{\rm un}$-isomorphism and $\psi_0(S) = S^*$ (and thus also $\psi_0(T) = T^*$).  Since $\psi_0$ restricted to $A$ is defined over $F$, we see that $\psi_0(A)$ is an $F$-split subtorus of $T^*$ and hence $\psi_0(A) \subseteq A^*$.    

Let $\tilde{\sigma}$ denote any lift in $\Gamma$ of the Frobenius element $\sigma \in {\rm Gal}(F^{\rm un}/F)$.  We may write
$$
\psi_0 \circ \tilde{\sigma}(\psi_0)^{-1} = \psi_0 \circ \sigma(\psi_0)^{-1} = {\rm Int}(m^*_{\sigma})$$ 
for an element $m^*_{\sigma} \in N_{M^*}(S^*)(F^s)$ whose image in $M^*_{\rm ad}(F^s)$ is well-defined.  As operators on $X_*(T^*) = X^*(\widehat{T^*})$, we may write
\begin{equation} \label{first_w_sigma}
\psi_0 \circ \sigma(\psi_0)^{-1} = w^*_\sigma
\end{equation}
for a well-defined element $w^*_\sigma \in W(M^*,S^*)(F^{\rm un})$.  Denote by $w_\sigma$ the preimage under the isomorphism $\psi_0: W(M,S)(F^{\rm un}) ~ \widetilde{\rightarrow} ~ W(M^*,S^*)(F^{\rm un})$ of $w^*_\sigma$.  Then (\ref{first_w_sigma}) translates into the equality
\begin{equation} \label{w_sigma}
\sigma \circ \psi^{-1}_0 \circ (\sigma^*)^{-1} \circ \psi_0 = w_\sigma
\end{equation} 
of operators on $X_*(T) = X^*(\widehat{T})$.  In defining $w_\sigma \in W(M,S)$, we fixed the objects $A$ and $S$ (needed to specify the ambient group $W(M,S)$) and along the way we also chose several additional objects: $P$, $A^*$, $S^*$, $B^*$, and an element $\psi_0 \in \Psi_M$ such that $\psi_0(S) = S^*$ and $\psi_0: M \rightarrow M^*$ is $F^{\rm un}$-rational.  It is straightforward to check that the element $w_\sigma \in W(M,S)$ is independent of all of these additional choices.

\subsection{Inner forms of split groups}
 
In this subsection we assume $G^*$ is $F$-split.  Then $A^* = S^* = T^*$, and $G_{F^{\rm un}}$ and $M_{F^{\rm un}}$ are split groups.  In particular, the relative Weyl group $W(M^*,S^*)$ coincides with the absolute Weyl group $W(M^*,T^*)$.  Using $\psi_0$ as above, we may regard $w_\sigma$ as an element of $W(M,S) = W(M,T)^I = W(\widehat{M}, \widehat{T})^I$.

For the next lemma, we need to recall the notion of cuspidal elements of Weyl groups.  Let $(W,S)$ be any Coxeter group with a finite set $S$ of simple reflections.  We say $w \in W$ is {\em cuspidal} if every conjugate of $w$ is elliptic, that is, every conjugate $w'$ has the property that any reduced expression for $w'$ contains every element of $S$.  Note that the identity element of $W$ is cuspidal if and only if $S = \emptyset$, in which case $W$ itself is trivial.

\begin{lemma} \label{w_sigma_lem_split}
\begin{enumerate}
\item[(a)]  The element $w_\sigma$ is a cuspidal element of the absolute Weyl group $W(M,T)$ of $M$.  
\item[(b)] The group $M$ is of type $A$ and the element $w_\sigma$ is a Coxeter element of $W(M,T)$.
\item [(c)] We have the equality $Z(\widehat{M}) =  \widehat{T}^{w_\sigma}$.
\end{enumerate}
\end{lemma}

\begin{proof}
Part (a).  We may assume $M \neq T$ and hence $W(M,T)$ is not trivial.  Suppose the assertion is false.  Then there is a notion of simple positive root for $M,T$ and a corresponding Coxeter group structure on $W(M,T)$, for which $w_\sigma$ is not an elliptic element.  Let $s_i$ denote a simple reflection in $W(M,T)$ which does not appear in a reduced expression for $w_\sigma$.  Then the corresponding fundamental coweight $\lambda_i \in X_*(T/Z(M))$ for $M_{\rm ad}$ is fixed by $w_\sigma$.  It is also fixed by $\psi^{-1}_0 \circ (\sigma^*)^{-1} \circ \psi_0$.  Thus by (\ref{w_sigma}) $\lambda_i$ is fixed by $\sigma$, and $\lambda_i(\mathbb G_m)$ is an $F$-split torus in $M_{\rm ad}$.  This contradicts the fact that $M_{\rm ad}$ is anisotropic over $F$.

Part (b).  Since every anisotropic $F$-group is type A (cf. Kneser \cite{Kn} and Bruhat-Tits \cite{BT3}, 4.3), the group $M$ is type $A$.  For type A groups, every cuspidal element in the Weyl group is Coxeter, as may be seen using cycle decompositions of permutations.  Thus, the cuspidal element $w_\sigma$ is a Coxeter element of $W(M,T)$. 

Part (c).  It is enough to prove the following statement: if $\mathcal G$ is a type A connected reductive complex group with maximal torus $\mathcal T$, and if $w \in W(\mathcal G, \mathcal T)$ is a Coxeter element, then $Z(\mathcal G) = \mathcal T^w$.  First, if $\mathcal G = {\rm PGL}_n$, a simple computation shows that $\mathcal T^w = 1 = Z(\mathcal G)$.  Since $\mathcal G_{\rm ad}$ is a product of projective linear groups and $w$ corresponds to a product of Coxeter elements, this also handles the case of adjoint groups.   In the general case, note that an element $t \in \mathcal T^w$ maps to $(\mathcal T_{\rm ad})^w = 1$ in $\mathcal G_{\rm ad}$, hence $t \in {\rm ker}(\mathcal G \rightarrow \mathcal G_{\rm ad}) = Z(\mathcal G)$.
\end{proof}

\begin{cor} \label{G*_split_cor}
If $G$ is an inner form of an $F$-split group, then 
$$\Lambda_M = X^*(Z(\widehat{M})) = X^*(\widehat{T}^\sigma) = X_*(T)_\sigma.$$
\end{cor}

\begin{proof}
The element $\sigma^*$ acts trivially on $Z(\widehat{M}) \hookrightarrow \widehat{T^*}$, since $T^*$ is $F$-split.   Moreover $w_\sigma \in W(M,T)$ acts trivially on $X^*(Z(\widehat{M}))$.  Then using 
(\ref{w_sigma}) it follows that $\sigma$ acts trivially on $X^*(Z(\widehat{M}))_I = X^*(Z(\widehat{M}))$.  This proves the first equality.

The second equality follows similarly using Lemma \ref{w_sigma_lem_split},(c), and the third equality is apparent.
\end{proof}

\section{The transfer homomorphism} \label{transfer_sec}

Now we return to the conventions and notation of subsection \ref{passing_subsec}.  Let $\mathcal A^S_L$ (resp.~$\mathcal A^{S^*}_L$) denote the apartment of $\mathcal B(G(L))$ (resp.~$\mathcal B(G^*(L))$) corresponding to $S$ (resp.~$S^*$).  The twisting $\psi_0$ gives an isomorphism $X_*(S)_{\mathbb R} \rightarrow 
X_*(S^*)_{\mathbb R}$ of the real vector spaces underlying these apartments.  Let $K$ (resp.~$K^*$) denote a special maximal parahoric subgroup of $G(F)$ (resp.~$G^*(F)$) corresponding to a special vertex in $(\mathcal A^S_L)^\sigma$ (resp.~$(\mathcal A^{S^*}_L)^{\sigma^*}$).  Then our goal is to define a canonical algebra homomorphism
$$
t: \calH_{K^*}(G^*) \rightarrow \calH_{K}(G).
$$

We expect $t$ will play a role in the study of Shimura varieties with parahoric level structure and in some related problems in p-adic harmonic analysis.  These issues will be addressed on another occasion.

\subsection{Relating the relative Weyl groups for $G^*$ and $G$}

\begin{prop} \label{relating_Weyl_grps}
Any twist $\psi_0 \in \Psi_M$ induces a map
$$
W(G,A) \rightarrow W(G^*,A^*)/W(M^*,A^*).
$$
\end{prop}

\begin{proof}
For $w \in W(G,A)$, choose a lift $n \in N_G(S)^\sigma$ (cf.~Lemma \ref{taut_hom}).  Write
$$
\sigma \circ \psi_0^{-1} \circ (\sigma^*)^{-1} \circ \psi_0 = {\rm Int}(m_\sigma)
$$
for an element $m_\sigma \in N_M(S)(F^s)$.  Set $m_* = \psi_0(\sigma^{-1}(m_\sigma)) \in N_{M^*}(S^*)(F^s)$.  Using $\sigma(n) = n$ and the fact that $\psi_0(n)$ normalizes $M^*$, we obtain
\begin{align*}
(\sigma^*)^{-1}(\psi_0(n)) &= m_* \, \psi_0(n) \, m^{-1}_* \\
&= \psi_0(n) \cdot ( \psi_0(n)^{-1}m_* \psi_0(n) m_*^{-1}) \\
&\in \psi_0(n) \, N_{M^*}(S^*).
\end{align*}
Thus $n \mapsto \psi_0(n)$ induces a well-defined map
$$
W(G,A) \rightarrow \Big(W(G^*,S^*)/W(M^*,S^*)\Big)^{\sigma^*}.
$$
The natural map $W(G^*,S^*)^{\sigma^*} \rightarrow \Big(W(G^*,S^*)/W(M^*,S^*)\Big)^{\sigma^*}$ is surjective.  Indeed, the choice of an $F$-rational Borel subgroup of $G^*$ containing $T^*$ gives us a notion of length on $W(G^*,S^*)$ which is preserved by $\sigma^*$, so that the minimal-length representatives of $\sigma^*$-fixed cosets in $W(G^*,S^*)/W(M^*,S^*)$ are fixed by $\sigma^*$.  It follows that
$$
W(G^*,S^*)^{\sigma^*}/W(M^*,S^*)^{\sigma^*} = \Big(W(G^*,S^*)/W(M^*,S^*)\Big)^{\sigma^*}.
$$
Thus, we have a well-defined map 
$$
W(G,A) \rightarrow W(G^*,S^*)^{\sigma^*}/W(M^*,S^*)^{\sigma^*} = W(G^*,A^*)/W(M^*,A^*)
$$
(cf. Remark \ref{qs_Weyl_bijection}).
\end{proof} 

\subsection{Definition of $t: \calH_{K^*}(G^*) \rightarrow \calH_K(G)$}

The isomorphism $$\widehat{\psi_0}: Z(\widehat{M^*}) ~ \widetilde{\rightarrow} ~  Z(\widehat{M})$$ is Galois-equivariant.  
Combined with the canonical inclusion $Z(\widehat{M^*}) \hookrightarrow \widehat{T^*}$ we see that $\widehat{\psi_0}$ induces a homomorphism 
\begin{equation} \label{psi_0_inv}
\psi_0: X^*(\widehat{T^*})^{\sigma^*}_I  \rightarrow X^*(Z(\widehat{M}))^\sigma_I.
\end{equation}

Since $W(M^*,A^*)$ induces the trivial action on $Z(\widehat{M^*})$, it follows using Proposition \ref{relating_Weyl_grps} that (\ref{psi_0_inv}) is equivariant with respect to the map $W(G,A) \rightarrow W(G^*,A^*)/W(M^*,A^*)$, in an obvious sense.  We thus get an algebra homomorphism
\begin{equation} \label{induced_psi_0}
\psi_0: \mathbb C[X^*(\widehat{T^*})^{\sigma^*}_I]^{W(G^*,A^*)} \rightarrow \mathbb C[X^*(Z(\widehat{M}))^\sigma_I]^{W(G,A)}.
\end{equation}
Since $\Psi_M$ is a torsor for $M^*_{\rm ad}$, one can check that this homomorphism is independent of the choice of $\psi_0$ in $\Psi_M$.  In fact it depends only on the choice of $A$ and $A^*$.  Therefore it makes sense to denote it by $t_{A^*,A}$ in what follows.  
It is easy to check that this homomorphism is surjective when $G^*$ is split over $F$.

\begin{defn} \label{t_defn}
Fix $A$ and $A^*$ as above.  Define $t: \calH_{K^*}(G^*) \rightarrow \calH_K(G)$ to be the unique homomorphism making the following diagram commute
$$
\xymatrix{
\calH_{K^*}(G^*) \ar[r]^{t} \ar[d]_{\wr} & \calH_K(G) \ar[d]_{\wr}  \\
 \mathbb C[X^*(\widehat{T^*})^{\sigma^*}_I]^{W(G^*,A^*)} \ar[r]^{t_{A^*,A}} & \mathbb C[X^*(Z(\widehat{M}))^\sigma_I]^{W(G,A)},
}$$
where the vertical arrows are the Satake isomorphisms.
\end{defn}

Obviously $t$ depends on $K$ and $K^*$.  It is easy to see that $t$ is independent of all other choices used in its construction.  Also, if $G^*$ is split over $F$, $t$ is surjective.

\subsection{Compatibilities with constant term homomorphisms}

Let $A$, $A^*$, $K$, and $K^*$ be fixed as above.  Let $H$ be a semi-standard $F$-Levi subgroup of $G$; this means that $H = C_G(A_H)$ for some subtorus $A_H \subseteq A$.  Let $H^*$ be a semi-standard $F$-Levi subgroup of $G^*$, so that $H^* = C_{G^*}(A^*_{H^*})$ for a subtorus $A^*_{H^*} \subset A^*$.  We have $M \subseteq H$ and $T^* \subseteq H^*$.  Let us suppose that some inner twist $G \rightarrow G^*$ restricts to give an inner twist $H \rightarrow H^*$.  

For example, for any $\psi_0 \in \Psi_M$ as above, we could take $A_H$ to be any subtorus of $A$ and set $A^*_{H^*} = \psi_0(A_H)$ (recalling that $\psi_0(A) \subseteq A^*$).  

Choose any $F$-rational parabolic subgroup $P_H = H N_H$ of $G$ with unipotent radical $N_H$ which contains $H$ as a Levi factor.  Recall the constant term map $c^G_H : \calH_K(G) \rightarrow \calH_{H \cap K}(H)$, which is defined by
\begin{equation} 
c^G_H(f)(h) = \delta^{1/2}_{P_H}(h)\int_{N_H(F)} f(hn) \, dn,
\end{equation}
for $h \in H(F)$, where the Haar measure $dn$ on $N_H(F)$ gives $N_H(F) \cap K$ measure 1.  We have a commutative diagram
\begin{equation} \label{const_term_comp}
\xymatrix{
\calH_K(G) \ar[r]^{\sim} \ar[d]_{c^G_H} & \mathbb C[\Lambda_M]^{W(G,A)} \ar[d] \\
\calH_{H \cap K}(H) \ar[r]^{\sim} & \mathbb C[\Lambda_M]^{W(H,A)},}
\end{equation}
where the horizontal arrows are the Satake isomorphisms, and the right vertical arrow is the inclusion homomorphism.  It follows that $c^G_H$ is an injective algebra homomorphism which is independent of the choice of $F$-rational parabolic subgroup $P_H \subseteq G$ which contains $H$ as a Levi factor.

The following proposition is proved using (\ref{const_term_comp}) and the definitions.

\begin{prop} \label{constant_term}
The following diagram commutes:
$$
\xymatrix{
\calH_{K^*}(G^*) \ar[r]^t \ar[d]_{c^{G^*}_{H^*}} & \calH_K(G) \ar[d]_{c^G_M} \\
\calH_{H^* \cap K^*}(H^*) \ar[r]^{t} & \calH_{H \cap K}(H).}
$$
\end{prop}
\qed

Taking $H = M$, the diagram shows that in order to compute $t$, it is enough to compute it in the case where $G_{\rm ad}$ is anisotropic.  In that case, if $f \in \calH_{K^*}(G^*)$, the function $t(f)$ is given by summing $f$ over the fibers of the Kottwitz homomorphism $k_{G^*}(F)$.

\small
\bigskip
\obeylines
\noindent
University of Maryland
Department of Mathematics
College Park, MD 20742-4015 U.S.A.
email: tjh@math.umd.edu, srostami@math.umd.edu

\end{document}